\documentclass[12pt]{article}
\usepackage{amsmath,amssymb,amsthm,amscd}
\usepackage[shortlabels]{enumitem}
\usepackage{mathtools}
\mathtoolsset{showonlyrefs}
\usepackage[all]{xy}
\usepackage[headinclude,headlines=1]{typearea}
\usepackage[dvipdfmx,colorlinks,setpagesize=false,linktocpage=true]{hyperref}
\usepackage{titletoc}
\titlecontents{section}[3em]{}{\contentslabel{2em}}{}{\hfill \contentspage}{}
\usepackage[nottoc]{tocbibind}
\usepackage{color}	

\newtheorem{thm}{Theorem}[section]
\newtheorem*{thm*}{Theorem}

\newtheorem{lem}[thm]{Lemma}

\newtheorem*{claim*}{Claim}

\theoremstyle{definition}

\newtheorem*{notn*}{Notation}
\newtheorem*{term*}{Terminology}

\newtheorem{para}[thm]{}

\setlist[enumerate]
{leftmargin=56pt,labelsep=8pt,itemsep=0pt,label=\upshape{(\thethm.\arabic*)}}

\numberwithin{equation}{thm}

\newcommand{\RomI}{\uppercase\expandafter{\romannumeral 1}}
\newcommand{\RomII}{\uppercase\expandafter{\romannumeral 2}}

\newcommand{\bC}{\mathbb C}

\newcommand{\bZ}{\mathbb Z}
\newcommand{\bpZ}{\mathbb Z_{> 0}}

\newcommand{\bR}{\mathbb R}
\newcommand{\bpR}{\mathbb R_{> 0}}
\newcommand{\bV}{\mathbb V}

\newcommand{\li}{1, 2, \dots, l}
\newcommand{\sli}{\{\li\}}
\newcommand{\cA}{\mathcal A}

\newcommand{\cE}{\mathcal E}
\newcommand{\cF}{\mathcal F}
\newcommand{\cG}{\mathcal G}
\newcommand{\cH}{\mathcal H}

\newcommand{\cO}{\mathcal O}
\newcommand{\cV}{\mathcal V}
\newcommand{\del}{\partial}
\newcommand{\snc}{simple normal crossing }

\DeclareMathOperator{\cok}{Coker}
\DeclareMathOperator{\image}{Image}
\DeclareMathOperator{\kernel}{Ker}
\DeclareMathOperator{\pd}{\Delta^{\!\ast}}
\DeclareMathOperator{\rank}{rank}
\DeclareMathOperator{\shom}{\cH {\it om}}
\DeclareMathOperator{\sEnd}{\cE {\it nd}}

\begin{document}
\title
{A remark on the nilpotent orbit theorem
for unipotent complex variations of Hodge structure}
\author{Taro Fujisawa \\
 \\
Tokyo Denki University \\
e-mail: fujisawa@mail.dendai.ac.jp}
\date{\today}

\maketitle

\begin{abstract}
We present a new proof of a part of the nilpotent orbit theorem
for unipotent complex variations of Hodge structure.
In our proof, the $L^2$ extension theorem of Ohsawa-Takegoshi type
plays an essential role.
\end{abstract}

\tableofcontents

\section{Introduction}

\begin{para}
Let $X$ be a complex manifold,
$D$ a \snc divisor on $X$
and $(\mathbb{V}, F, Q)$
a complex variation of Hodge structure
on $X \setminus D$.
(For the definition of
a complex variation of Hodge structure,
see \cite[1.1]{DengNOT}.)
Throughout this paper,
we assume that $(\bV, F, Q)$ is unipotent,
that is, all the monodromy automorphisms
along the irreducible components of $D$
are unipotent.
The canonical extension of $\cO_{X \setminus D} \otimes \bV$
in the sense of Deligne {\rm (cf. \cite[\RomII, \S 5]{DeligneED})}
is denoted by $\cV$.
\end{para}

The following is a special case of \cite[Theorem 2.5]{DengNOT}
and known as a part of Schmid's nilpotent orbit theorem in \cite{Schmid}
for the case of $\bR$-Hodge structure:

\begin{thm*}[=Theorem \ref{thm:2}]
In the situation above,
the filtration $F$ on $\cO_{X \setminus D} \otimes_{\bC} \mathbb{V}$
extends to a filtration $F$ on $\cV$.
\end{thm*}

\begin{para}
The proof of the theorem in \cite{DengNOT} heavily relies
on the enormous works of Mochizuki
on the prolongation of acceptable bundles
(cf. \cite[1.3--1.5]{DengNOT}).
The aim of this paper is to present another proof of the theorem,
which is independent of Mochizuki's works.
The key ingredients of our proof are
Simpson's ``basic estimate''
and the $L^2$ extension theorem of Ohsawa-Takegoshi type.
\end{para}

\begin{para}
Simpson's basic estimate in \cite{SimpsonHBNCC}
is reproved in \cite[\S58. Corollary]{SabbahSchnell}
in a sharp and very explicit form.
Then it becomes clear that the constant involved in the estimate
depends only on the rank of the variation.
Therefore it can be directly generalised to the case of $\dim X \ge 2$
although only the case of $\dim X=1$ is treated in \cite{SabbahSchnell}.
Then, inspired by the work of Deng \cite{DengNOT},
we are led to a characterization of holomorphic sections of $\cV$
via the $L^2$-norm on $X \setminus D$
in Lemma \ref{lem:9}.
\end{para}

\begin{para}
The proof of the main theorem in this paper
proceeds by induction on $\dim X$.
For the case of $\dim X=1$,
the argument here is similar to arguments
in \cite{CornalbaGriffithsAC}
(see (9.9) Proposition, its Corollary and Remark after the corollary).
The use of the $L^2$ extension theorem of Ohsawa-Takegoshi type
substantially simplifies the arguments.
In the induction step,
the $L^2$ extension theorem plays an essential role again.
In order to extend the subbundle $F^p$ over the boundary $D$,
we have to find appropriate sections of $F^p$,
which can be extended over $D$.
Together with the characterization of holomorphic sections of $\cV$ above,
the $L^2$ extension theorem enables us to find such sections inductively
from the restriction to a smooth hypersurface of $X$.

As mentioned above,
the proof in this paper is logically independent of the works of Mochizuki.
However, it is largely inspired by his arguments
in \cite{MochizukiJDG} and \cite{MochizukiWHB}.
\end{para}

\begin{notn*}
On a complex manifold $X$,
we use the words ``vector bundle''
and ``locally free $\cO_X$-module of finite rank''
interchangeably.
For a vector bundle $\cV$,
a subbundle of $\cV$ is an $\cO_X$-submodule $\cF$ of $\cV$
such that $\cV/\cF$ is locally free of finite rank.
Then $\cF$ itself is a locally free $\cO_U$-module of finite rank.
A filtration $F$ on a vector bundle $\cV$
is always assumed to be a ``filtration by subbundles'',
that is, each $F^p\cV$ is assumed to be a subbundle of $\cV$.
\end{notn*}

\begin{term*}
We use symbols such as $\bZ, \bR, \bC$ etc. in the usual meanings.
The set of the positive real numbers is denoted by $\bpR$.
\end{term*}


\section{Extension of subbundles}
\label{sec:preliminaries}

\begin{para}
\label{para:4}
Throughout this section,
we work in the following situation:
Let $X$ be a complex manifold,
$S$ a nowhere dense closed analytic subset of $X$
and $U=X \setminus S$.
The open immersion $U \hookrightarrow X$ is denoted by $j$.
Moreover, $(\cV, \cF)$ be a pair
consisting of a locally free $\cO_X$-module of finite rank $\cV$
and a subbundle $\cF$ of $j^{-1}\cV$.
\end{para}

\begin{lem}
\label{lem:2}
In the situation above,
the canonical morphism
$\cV \longrightarrow j_{\ast}j^{-1}\cV$
is injective.
\end{lem}
\begin{proof}
Since $\cV$ is locally free, we may assume $\cV=\cO_X$.
Then we obtain the conclusion
because $U$ is dense in $X$.
\end{proof}

\begin{para}
\label{para:2}
By the lemma above, we have two injections
$\cV \hookrightarrow j_{\ast}j^{-1}\cV$,
$j_{\ast}\cF \hookrightarrow j_{\ast}j^{-1}\cV$
and obtain an $\cO_X$-submodule
$j_{\ast}\cF \cap \cV \subset j_{\ast}j^{-1}\cV$.
\end{para}

\begin{lem}
\label{lem:1}
For a pair $(\cV, \cF)$ as in \ref{para:4},
the following is equivalent:
\begin{enumerate}
\item
\label{item:5}
There exists a subbundle $\overline{\cF} \subset \cV$
such that
$j^{-1}\overline{\cF}=\cF$.
\item
\label{item:6}
$j_{\ast}\cF \cap \cV$ is a subbundle of $\cV$.
\end{enumerate}
Moreover, the extension $\overline{\cF}$
is unique if it exists.
\end{lem}
\begin{proof}
Trivially \ref{item:6} implies \ref{item:5}.
Conversely, we assume \ref{item:5}.
Then $j_{\ast}\cF \cap \cV \subset \overline{\cF}$
by \cite[Lemma 5.1]{Fujino-Fujisawa}.
On the other hand,
the canonical morphism
$\overline{\cF} \longrightarrow j_{\ast}j^{-1}\overline{\cF}=j_{\ast}\cF$
gives us an inclusion $\overline{\cF} \subset j_{\ast}\cF \cap \cV$.
Therefore $\overline{\cF}=j_{\ast}\cF \cap \cV$.
This equality also proves the uniqueness of $\overline{\cF}$
(cf. \cite[Corollary 5.2]{Fujino-Fujisawa}).
\end{proof}

\begin{para}
\label{para:3}
Let $Y$ be a closed submanifold of $X$
such that $T=Y \cap S \subset Y$ is nowhere dense again.
The canonical closed immersion $Y \hookrightarrow X$ is denoted by $i$
and the open immersion
$U \cap Y=Y \setminus T \hookrightarrow Y$ is denoted by $k$.
For a pair $(\cV, \cF)$ as in \ref{para:4},
we obtain a pair $(i^{\ast}\cV, (i|_{U \cap Y})^{\ast}\cF)$
which satisfy the same condition as in \ref{para:4}.
We set $\cV_Y=i^{\ast}\cV$
and $\cF_{U \cap Y}=(i|_{U \cap Y})^{\ast}\cF$.
The canonical morphism
$\cF \longrightarrow (i|_{U \cap Y})_{\ast}\cF_{U \cap Y}$
induces a morphism
$j_{\ast}\cF \longrightarrow i_{\ast}k_{\ast}\cF_{U \cap Y}$.
Thus we obtain
$j_{\ast}\cF \cap \cV \longrightarrow i_{\ast}(k_{\ast}\cF_{U \cap Y} \cap \cV_Y)$.
\end{para}

\begin{lem}
\label{lem:4}
In the situation \ref{para:4} and \ref{para:3},
the pair $(\cV, \cF)$ satisfies \ref{item:5}
on a neighborhood of $Y$
if the following two conditions are satisfied:
\begin{enumerate}
\item
\label{item:8}
The canonical morphism
$j_{\ast}\cF \cap \cV \longrightarrow i_{\ast}(k_{\ast}\cF_{U \cap Y} \cap \cV_Y)$
is surjective.
\item
\label{item:9}
$(\cV_Y, \cF_{U \cap Y})$ satisfies \ref{item:5}.
\end{enumerate}
\end{lem}
\begin{proof}
We set $r=\rank \cF$ and $s=\rank (j^{-1}\cV/\cF)$.
Take a point $x \in Y$
and set $\bC(x)=\cO_{X,x}/\frak{m}=\cO_{Y,x}/\frak{m}\cO_{Y,x}$,
where $\frak{m}$ is the maximal ideal of $\cO_{X,x}$.
Since $(\cV_Y, \cF_{U \cap Y})$ satisfies
the equivalent conditions \ref{item:5} and \ref{item:6},
the canonical morphism
\begin{equation}
\label{eq:3}
\bC(x) \otimes (k_{\ast}\cF_{U \cap Y} \cap \cV_Y)_x
\longrightarrow
\bC(x) \otimes \cV_{Y,x}
\end{equation}
is injective.
From the assumption \ref{item:8},
we can take local sections
$f_1, \dots, f_r \in j_{\ast}\cF \cap \cV$
whose images
form a base of the $\bC$-vector space
$\bC(x) \otimes (k_{\ast}\cF_{U \cap Y} \cap \cV_Y)_x$.
By shrinking $X$,
the sections $f_1, \dots, f_r$
give us a morphism of $\cO_X$-modules
$\varphi \colon \cO^{\oplus r}_X \longrightarrow \cV$.
Because the images of $f_1, \dots, f_r$
in $\bC(x) \otimes \cV_{X,x}=\bC(x) \otimes \cV_{Y,x}$
is linearly independent
by the injectivity of the morphism \eqref{eq:3},
the morphism $\varphi$ is injective
and $\cok(\varphi)$ is locally free of rank $s$
on a neighborhood of the point $x$.
Thus, by shrinking $X$ again,
$\overline{\cF}=\image(\varphi)$
is a subbundle of $\cV$ of rank $r$.
By definition, we have $\overline{\cF} \subset j_{\ast}\cF \cap \cV$.
Therefore we have $j^{-1}\overline{\cF}=\cF$
because both of them are subbundles of $j^{-1}\cV$ of the same rank
satisfying the inclusion $j^{-1}\overline{\cF} \subset \cF$.
\end{proof}

\section{Nilpotent orbit theorem}

\begin{para}
\label{para:1}
First, we work in the following local situation:
\begin{enumerate}
\item
$X=\Delta^n$
is the polydisc with the coordinates $(t_1, t_2, \dots, t_n)$,
\item
$D_i=\{t_i=0\}$ for $i=1,2, \dots, n$
is a divisor on $X$,
\item
$D=\sum_{i=1}^{l}D_i$
is a \snc divisor on $X$ and
\item
$U=X \setminus D$.
\end{enumerate}
The open immersion $U \hookrightarrow X$ is denoted by $j$.
The sheaf of the $C^{\infty}$-functions on $X$
is denoted by $\cA_X$.
We set
\begin{equation}
d\mu=(-2\sqrt{-1})^{-n}
dt_1 \wedge d\overline{t}_1 \wedge \cdots \wedge dt_n \wedge d\overline{t}_n,
\end{equation}
which is a $C^{\infty}$ $(n,n)$-form on $X$.
A $C^{\infty}$-function
\begin{equation}
\label{eq:11}
\varphi
=-\sum_{i=1}^{l}\log(-\log |t_i|^2)-\sum_{i=l+1}^{n}\log(1-|t_i|^2),
\end{equation}
gives us a K\"ahler form
\begin{equation}
\omega=\sqrt{-1}\del\overline{\del}\varphi
=\sqrt{-1}
(\sum_{i=1}^l\frac{dt_i \wedge d\overline{t}_i}{|t_i|^2(-\log |t_i|^2)^2}
+\sum_{i=l+1}^{n}\frac{dt_i \wedge d\overline{t}_i}{(1-|t_i|^2)^2})
\end{equation}
on $U$.
\end{para}

\begin{para}
\label{para:5}
Let $(\bV, F, Q)$ be a unipotent complex variation of Hodge structure on $U$.
and $\cV$ the canonical extension of $\cO_U \otimes \bV$.
The set of the multi-valued flat sections is denoted by $V$.
Then the situation is summarized as follows:
\begin{enumerate}
\item
There exist nilpotent endomorphisms $N_1, \dots, N_l$ on $V$
with $[N_i, N_j]=0$ for all $i,j \in \{1, 2, \dots, l\}$.
\item
$\cV=\cO_X \otimes_{\bC} V$.
\item
$\bV$ is given by
\begin{equation}
\bV=\exp(\sum_{i=1}^{l}\frac{1}{2\pi\sqrt{-1}}(\log t_i)N_i)(V)
\subset j^{-1}\cV=\cO_U \otimes V.
\end{equation}
(The morphism
$\exp(\sum_{i=1}^{l}(2\pi\sqrt{-1})^{-1}(\log t_i)N_i)$
is multi-valued, but the image of $V$ by the morphism is well-defined.)
\end{enumerate}
We define $k_i \in \bpZ$ by the property
\begin{equation}
N_i^{k_i} \not=0, N_i^{k_i+1}=0
\end{equation}
for each $i \in \sli$.
The filtration $F$ on $\cO_U \otimes \bV$
is regraded as a filtration on $j^{-1}\cV$
via the isomorphism $j^{-1}\cV \simeq \cO_U \otimes \bV$
and $F^p(j^{-1}\cV)$ and $F^p(\cO_U \otimes \bV)$
will be used interchangeably.
We set $\cE=\cA_X \otimes_{\bC}V=\cA_X \otimes_{\cO_X} \cV$
and $F^p(j^{-1}\cE)=\cA_X \otimes_{\cO_X}F^p(j^{-1}\cV)$
for $p \in \bZ$.
The Higgs field associated to $(\bV, F, Q)$
is denoted by
$\theta \in \Gamma(U, \cA_U^{1,0} \otimes \sEnd(j^{-1}\cE))$,
where $\cA_U^{1,0}$ denotes
the sheaf of the $C^{\infty}$ $(1,0)$-forms on $U$.
The Hodge metric on $j^{-1}\cE$
is denoted by $h$.
Then $h$ induces an Hermitian metric on $F^p(j^{-1}\cE)$
for all $p$,
which is denoted by the same letter $h$ by abuse of notation.
The norm induced by $h$ is denoted by $| \cdot |_h$.
The Hermitian metric $h$ induces
the conjugate linear isomorphism
$j^{-1}\cE \simeq j^{-1}\cE^{\ast}$,
where $\cE^{\ast}=\shom_{\cA_X}(\cE, \cA_X)$ is the dual of $\cE$.
Then $h$ induces Hermitian metrics on $j^{-1}\cE^{\ast}$
and $\sEnd(\j^{-1}\cE)\simeq j^{-1}\cE^{\ast} \otimes_{\cA_U} \j^{-1}\cE$
in the usual way.
We denote them by $h$ by abuse of notation.
\end{para}

The following lemma is proved
in \cite[\S58. Corollary]{SabbahSchnell}
for the case of $\dim X=1$.

\begin{lem}
\label{lem:7}
In the situation \ref{para:1} and \ref{para:5},
we set
$\theta=\sum_{i=1}^{n}\theta_i dt_i$
with $\theta_i \in \Gamma(U, \sEnd(j^{-1}\cE))$
for $i=1, 2, \dots, n$.
Then
\begin{equation}
h(\theta_i, \theta_i) \le \frac{C_0^2}{|t_i|^2(\log |t_i|^2)^2}
\end{equation}
for $i=\li$ and
\begin{equation}
h(\theta_i, \theta_i) \le \frac{C_0^2}{(1-|t_i|^2)^2}
\end{equation}
for $i=l+1, \dots, n$,
where $r=\rank \bV$ and $C_0=\sqrt{r(r+1)(r-1)/6}$.
\end{lem}
\begin{proof}
Since the constant involved in the estimate is explicitly given
in \cite{SabbahSchnell},
and depends only on the rank of $\bV$,
we obtain the conclusion easily.
\end{proof}

\begin{lem}
\label{lem:8}
In the situation \ref{para:1} and \ref{para:5},
there exists a constant $\lambda_0 >0$
such that
$he^{-\lambda \varphi}$ is Nakano semipositive on $F^p(\cO_U \otimes \bV)$
for each $p$ and for any $\lambda \ge \lambda_0$,
where $\varphi$ is the $C^{\infty}$-function defined in \eqref{eq:11}.
\end{lem}
\begin{proof}
By the lemma above
and by \cite[Lemma 1.4]{DengNOT},
we can apply Lemma 1.10 of \cite{DengHao} to $(F^p(\cO_U \otimes \bV), h)$
and obtain the conclusion.
\end{proof}

\begin{lem}
\label{lem:3}
In the situation \ref{para:1} and \ref{para:5},
for $f \in \Gamma(X, \cV)$
and for a compact subset $K \subset X$,
there exists $C \in \bpR$ such that
\begin{equation}
|f|_h \le C\prod_{i=1}^l(-\log |t_i|)^{C_0+k_i}
\end{equation}
on $K \cap U$.
\end{lem}
\begin{proof}
There exists a finite open cover $\{U_{\mu}\}$ of $U$
such that $\log t_i$ becomes a single valued holomorphic funtion
for all $i \in \sli$
on every $U_{\mu}$ by fixing a branch of the logarithm.
By setting
\begin{equation}
\label{eq:7}
\widetilde{v}_{\mu}
=\exp(\sum_{i=1}^{l}\frac{1}{2\pi\sqrt{-1}}(\log t_i)N_i)(v),
\end{equation}
we obtain a section
$\widetilde{v}_{\mu} \in \Gamma(U_{\mu}, \bV)$ for all $\mu$.
From Lemma \ref{lem:7},
there exists $C_{\mu} \in \bpR$
such that the inequality
\begin{equation*}
|\widetilde{v}_{\mu}|_h \le C_{\mu}\prod_{i=1}^l(-\log |t_i|)^{C_0}
\end{equation*}
holds on $K \cap U_{\mu}$
(see e.g. \cite[Chapter \RomII, Proposition 25]{GriffithsTTAG}).
Because we have
\begin{equation}
\label{eq:8}
v=\exp(-\sum_{i=1}^{l}\frac{1}{2\pi\sqrt{-1}}(\log t_i)N_i)\widetilde{v}_{\mu},
\end{equation}
on $U_{\mu}$
from \eqref{eq:7}
and because $N_1, \dots, N_l$ are mutually commutative nilpotent endomorphisms,
there exist $C'_{\mu} \in \bpR$
such that
\begin{equation}
\label{eq:10}
|v|_h^2 \le C'_{\mu}\prod_{i=1}^l(-\log |t_i|)^{C_0+k_i}
\end{equation}
holds on $K \cap U_{\mu}$.
Because $\{U_{\mu}\}$ is finite open covering of $U$,
we obtain the conclusion.
\end{proof}

\begin{lem}
\label{lem:9}
In the situation \ref{para:1} and \ref{para:5},
the following three conditions are equivalent
for $f \in \Gamma(U, j^{-1}\cV)$:
\begin{enumerate}
\item
\label{item:12}
There exists $F \in \Gamma(X, \cV)$
such that $F|_U=f$.
\item
\label{item:11}
For any compact subset $K$ of $X$,
we have
\begin{equation}
\int_{K \cap U} |f|^2_h\prod_{i=1}^l(-\log |t_i|)^{\alpha_i} d\mu < \infty
\end{equation}
for any $(\alpha_1, \dots, \alpha_l) \in \bR^l$.
\item
\label{item:3}
For any compact subset $K$ of $X$,
we have
\begin{equation}
\int_{K \cap U} |f|^2_h\prod_{i=1}^l(-\log |t_i|)^{\alpha_i} d\mu < \infty
\end{equation}
for some $(\alpha_1, \dots, \alpha_l) \in \bR^l$
with $\alpha_i > 2(C_0+k_i)$ for all $i$.
\end{enumerate}
\end{lem}
\begin{proof}
Using the polar coordinates
$t_1=r_1e^{i\theta_1}, \dots, t_l=r_le^{i\theta_l}$,
we have
\begin{equation}
d\mu=(\prod_{i=1}^lr_i)
dr_1 \wedge d\theta_1 \wedge \cdots \wedge dr_l \wedge d\theta_l
\wedge dt_{l+1} \wedge d\overline{t}_{l+1} \wedge \cdots \wedge
dt_n \wedge d\overline{t}_n
\end{equation}
and then conclude that \ref{item:12} implies \ref{item:11}.
The implication \ref{item:11} $\Rightarrow$ \ref{item:3} is trivial.
Now we prove \ref{item:3} $\Rightarrow$ \ref{item:12}.
Let $\{v_1, \dots, v_r\}$ be a base of $V$.
Then we have
\begin{equation}
f=\sum_{i=1}^{r}f_iv_i,
\end{equation}
where $f_i \in \Gamma(U, \cO_U)$.
It suffices to prove that
$f_i$ can be extended to a unique holomorphic function on $X$.
Thus we may assume $l=1$
by Riemann's extension theorem for holomorphic functions.
Moreover, by Fubini's Theorem,
we may assume $n=1$,
that is, $X=\Delta, D=\{0\}$ and $U=\pd$.
Then we use $t$, $k$ and $\alpha$
for $t_1$, $k_1$ and $\alpha_1$ for short.
We fix $i \in \{1, \dots, r\}$ and
will prove that $f_i$ can be extended to $\Delta$.
Without loss of generality, we may assume $i=r$.

Fix a compact subset $K=\{t \in \Delta \mid |t| \le R\} \subset \Delta$
for some $R \in \bR$ with $0 < R < 1$.
By the Gram-Schmidt orthonormalization,
we have an orthonormal base
$\{u_1, \dots, u_r\}$ of $j^{-1}\cE$
such that
\begin{equation}
\label{eq:1}
v_i=\sum_{j=1}^{i}a_{ij}u_j
\end{equation}
for $i=1, 2, \dots, r$,
where $a_{ij}$ is a $C^{\infty}$-function on $\pd$.
We consider the dual variation $(\bV^{\ast}, F, Q^{\ast})$
of $(\bV, F, Q)$
and take the dual bases
$\{v_1^{\ast}, \dots, v_r^{\ast}\}$ of $V^{\ast}$
for $\{v_1, \dots, v_r\}$
and $\{u_1^{\ast}, \dots, u_r^{\ast}\}$ of $j^{-1}\cE^{\ast}$
for $\{u_1, \dots, u_r\}$ respectively.
Then we have $u_r^{\ast}=a_{rr}v_r^{\ast}$ from \eqref{eq:1}.
By applying Lemma \ref{lem:3}
to $(\bV^{\ast}, F, Q^{\ast})$,
we obtain
\begin{equation}
\label{eq:12}
1=|u_r^{\ast}|^2_{h^{\ast}}
=|a_{rr}|^2|v_r^{\ast}|_{h^{\ast}}^2
\le C(-\log |t|)^{2(C_0+k)}|a_{rr}|^2
\end{equation}
on $K \cap \pd$ for some $C \in \bpR$,
where $h^{\ast}$ denotes the Hodge metric on $j^{-1}\cE^{\ast}$.
Since
\begin{equation}
f=\sum_{i}f_iv_i=\sum_{j=1}^{r-1}(\sum_{i>j}f_ia_{ij})u_j+f_ra_{rr}u_r,
\end{equation}
we have
\begin{equation}
|f|^2_h \ge |f_r|^2|a_{rr}|^2 \ge C^{-1}(-\log |t|)^{-2(C_0+k)}|f_r|^2
\end{equation}
on $K \cap \pd$ by \eqref{eq:12},
and then
\begin{equation}
\int_{K \cap \pd}(-\log |t|)^{\alpha-2(C_0+k)}|f_r|^2 d\mu
\le C_2\int_{K \cap \pd}|f|_h^2(-\log |t|)^{\alpha} d\mu < \infty
\end{equation}
for $\alpha > 2(C_0+k)$ by the assumption \ref{item:3}.
Therefore $f_r$ is a holomorphic function on $\Delta$
(cf. \cite[Lemma 3.4]{SchnellYang}).
\end{proof}

Now we are ready to prove the main result of this paper.

\begin{thm}
\label{thm:2}
Let $X$ be a complex manifold,
$D$ a \snc divisor on $X$,
and $(\bV, F, Q)$ a complex polarized variation of Hodge structure
on $X \setminus D$.
Assume that all the local monodromy automorphisms
along the irreducible components of $D$ are unipotent.
The canonical extension of $\cO_{X \setminus D} \otimes \bV$
is denoted by $\cV$.
Then the filtration $F$ on $\cO_{X \setminus D} \otimes \bV$
extends to a filtration on $\cV$.
\end{thm}
\begin{proof}
We set $U=X \setminus D$ and
use the notation as in \ref{para:4} and \ref{para:3}.
It suffices to prove that
$F^p(\cO_U \otimes \bV)$ can be extended to a subbundle of $\cV$
for each $p$.
Therefore, we fix $p \in \bZ$
and denote $F^p(\cO_U \otimes \bV)$ by $\cF$,
which is a subbundle of $j^{-1}\cV$.
Since the extension is unique by \ref{lem:1},
we may work in the local situation in \ref{para:1} and \ref{para:5}.
According to Lemma \ref{lem:8},
we fix $\lambda \in \bpR$
such that $he^{-\lambda \varphi}$
is Nakano semipositive on $\cF$.

First, we prove the theorem for the case of $\dim X=1$.
We use $t$ for $t_1$ as before.
Take a point $t_0 \in U=\pd$
and a base
$\{v^{(0)}_1, \cdots, v^{(0)}_s\}$ of $\bC(t_0) \otimes \cF_{t_0}$,
where $s=\rank \cF$.
Since $(U, \{t_0\})$ satisfies the condition (ab)
in \cite[Definition 1.1]{GuanZhouSL2EP},
we can apply Theorem 2.2 of \cite{GuanZhouSL2EP}
to $(\cF, he^{-\lambda \varphi})$
by setting $A=\log 4$, $c_A(t)=1$ and $\Psi=\log |t-t_0|^2$.
Thus there exist
$f_1, \dots, f_s \in \Gamma(U, \cF) \subset \Gamma(U, j^{-1}\cV)$
such that
\begin{equation}
\int_U |f_i|_{he^{-\lambda \varphi}}^2 d\mu
=\int_U |f_i|^2_h(-\log |t|^2)^{\lambda} d\mu
< \infty
\end{equation}
and $f_i(t_0)=v^{(0)}_i$ for all $i$.
Therefore $f_i$'s can be extended to
elements of $\Gamma(X, \cV)$ for all $i$
by Lemma \ref{lem:9}.
We use the same letter $f_i$ for the extension in $\Gamma(X, \cV)$
by abuse of notation.
The $\cO_X$-submodule of $\cV$ generated by $f_1, \dots, f_r$
is denoted by $\cF'$.
From $\cF' \subset \cV$, we know that $\cF'$ is torsion free
and then $\cF'$ is locally free of finite rank on $X$.
Moreover the rank of $\cF'$ is $s$
because of the property
that $\{f_1(t_0), \dots, f_s(t_0)\}$
is linearly independent
in $V=\bC(t_0) \otimes \cV_{t_0}$.
By shrinking $X$ sufficiently.
we may assume that
$j^{-1}\cF'$ is a subbundle of $j^{-1}\cV$.
Then $j^{-1}\cF'=\cF$
because both of them are subbundles of $j^{-1}\cV$ of the same rank
with $j^{-1}\cF' \subset \cF$.
We consider the quotient $\cV/\cF'$,
which is locally free on $U$ as above.
The image of the canonical morphism
\begin{equation}
\cV/\cF' \longrightarrow (\cV/\cF')^{\ast \ast}
\end{equation}
is denoted by $\cG$,
where
$(\cV/\cF')^{\ast \ast}=\shom_{\cO_X}(\shom_{\cO_X}(\cV/\cF', \cO_X), \cO_X)$
is the doable dual of $\cV/\cF'$.
Then $\cG$ is a locally free $\cO_X$-module of finite rank
because $\cG \subset (\cV/\cF')^{\ast \ast}$ is torsion free.
Moreover, we have the surjection $\cV \longrightarrow \cG$
such that
$j^{-1}\kernel(\cV \longrightarrow \cG)=\cF$.
Therefore, by setting
$\overline{\cF}=\kernel(\cV \longrightarrow \cG)$,
we obtain a subbundle $\overline{\cF}$ of $\cV$
such that
$j^{-1}\overline{\cF}=\cF$.

Now, we proceed by induction on $n=\dim X$.
We assume $n \ge 2$
and take a closed submanifold $Y$ of $X$ defined by $Y=\{t_1=t_2\}$.
Then $Y \simeq \Delta^{n-1}$ with the coordinate $(t, t_3, \dots, t_n)$
by setting $t=t_1=t_2$ on $Y$.
We use the notation such as $i: Y \hookrightarrow X$,
$k: U \cap Y \hookrightarrow Y$,
$\cV_Y=i^{\ast}\cV$ and $\cF_{U \cap Y}$
as in \ref{para:3}.
Note that the restriction
$((i|_{U \cap Y})^{-1}\bV, (i|_{U \cap Y})^{\ast}F, (i|_{U \cap Y})^{-1}Q)$
of $(\bV, F, Q)$
is a unipotent complex variation of Hodge structure on $U \cap Y$.
Then, by the induction hypothesis,
the pair $(\cV_Y, \cF_{U \cap Y})$ satisfies \ref{item:5}.
Therefore it suffices to prove \ref{item:8}
by Lemma \ref{lem:4}.
For
\begin{equation}
f \in \Gamma(Y, k_{\ast}\cF_{U \cap Y} \cap \cV_Y)
= \Gamma(U \cap Y, \cF_{U \cap Y}) \cap \Gamma(Y, \cV_Y),
\end{equation}
we may assume
\begin{equation}
\int_{U \cap Y}|f|_{he^{-\lambda \varphi}}^2 d\mu
=
\int_{U \cap Y}|f|_h^2
(-\log |t|^2)^{2\lambda}\prod_{i=3}^l(-\log |t_i|^2)^{\lambda}
\prod_{i=l+1}^n(1-|t_i|^2)^{\lambda} d\mu
< \infty
\end{equation}
by Lemma \ref{lem:9} and by shrinking $X$.
Since $(U, U \cap Y)$ satisfies the condition (ab)
in \cite[Definition 1.1]{GuanZhouSL2EP} again,
we obtain
$F \in \Gamma(U, \cF)$
such that
$F|_{U \cap Y}=f$ with
\begin{equation}
\label{eq:13}
\int_U|F|_{he^{-\lambda \varphi}}^2 d\mu
=\int_U|F|_h^2\prod_{i=1}^l(-\log |t_i|^2)^{\lambda}
\prod_{i=l+1}^n(1-|t_i|^2)^{\lambda} d\mu
< \infty
\end{equation}
by applying Theorem 2.2 of \cite{GuanZhouSL2EP} as before.
The estimate \eqref{eq:13} implies that
\begin{equation}
F \in \Gamma(U, \cF) \cap \Gamma(X, \cV)
=\Gamma(X, j_{\ast}\cF \cap \cV)
\end{equation}
by Lemma \ref{lem:9}.
Because $k_{\ast}\cF_{U \cap Y} \cap \cV_Y$ is a subbundle of $\cV$
by the induction hypothesis,
we conclude \ref{item:8}
by applying the process above to a local base of
$k_{\ast}\cF_{U \cap Y} \cap \cV_Y$.
\end{proof}

\providecommand{\bysame}{\leavevmode\hbox to3em{\hrulefill}\thinspace}
\providecommand{\MR}{\relax\ifhmode\unskip\space\fi MR }
\providecommand{\MRhref}[2]{%
  \href{http://www.ams.org/mathscinet-getitem?mr=#1}{#2}
}
\providecommand{\href}[2]{#2}

\end{document}